\documentclass[11pt]{amsart}
\usepackage{amssymb,bm}
\usepackage{amsmath}
\usepackage{amsfonts}
\usepackage{amsthm}
\usepackage{color}
\usepackage{graphics,graphicx,url}
\usepackage{enumerate}
\usepackage{comment}
\usepackage{subcaption}
\usepackage{thmtools}
\usepackage{thm-restate}

\usepackage{centernot}

\usepackage{breqn} 

\usepackage[a4paper]{geometry}
\usepackage{xcolor}

\newcommand\sppt{\mathrm{isupp}}

  \definecolor{green4}{rgb}{0.247,0.498,0.208}

\usepackage{tikz}
\usetikzlibrary{calc,decorations.pathmorphing,patterns,arrows,decorations.markings,positioning}
\usetikzlibrary{automata,chains,fit,shapes}

\usepackage{graphicx,pdflscape,rotating,float}
\usepackage{enumitem}
\newcommand{\bi}{\begin{itemize}}
\newcommand{\ei}{\end{itemize}}
\newcommand{\be}{\begin{enumerate}}
\newcommand{\ee}{\end{enumerate}}
\newcommand{\bc}{\begin{center}}
\newcommand{\ec}{\end{center}}
\newcommand{\bt}{\begin{tabular}}
\newcommand{\et}{\end{tabular}}
\newcommand{\ba}{\begin{array}}
\newcommand{\ea}{\end{array}}

\newcommand{\cC}{\varsigma}

\newenvironment{proof*}[1]
  {\renewcommand{\proofname}{#1}%
   \begin{proof}}
  {\end{proof}}

\newtheorem{thm}{Theorem}[section]

\theoremstyle{definition}

\newtheorem{defn}[thm]{Definition}

\usepackage[normalem]{ulem}

\newcommand\Z{\mathbb Z}
\newcommand\N{\mathbb N}

\newcommand\h{x}

\renewcommand{\geq}{\geqslant} \renewcommand{\leq}{\leqslant}

\title[Cayley automatic is closed under taking wreath product with virtually cyclic]{Being Cayley automatic is closed under taking \\wreath product with virtually cyclic groups}

\author[D. Berdinsky]{Dmitry Berdinsky\textsuperscript{a,b}}
\address{\textsuperscript{a}Department of Mathematics,
		Faculty of Science,
		Mahidol
		University, Bangkok, 10400, Thailand
		\textsuperscript{b}Centre of Excellence in Mathematics,
		Commission on Higher Education, Bangkok, 10400, Thailand}
\email{berdinsky@gmail.com}

\author[M. Elder]{Murray Elder}\thanks{The second author is supported by Australian Research Council grant DP160100486}
\address{School of Mathematical and Physical Sciences, University of Technology Sydney, Ultimo, NSW 2007, Australia}
\email{murray.elder@uts.edu.au}

\author[J. Taback]{Jennifer Taback}\thanks{The third author acknowledges support from Simons Foundation grant 31736 to Bowdoin College.}
\address{Department of Mathematics,
Bowdoin College, 8600 College Station, Brunswick, ME 04011, USA} \email{jtaback@bowdoin.edu}

\date{\today}


\begin{document}

\begin{abstract}
We extend work of the first author and Khoussainov to show that being Cayley automatic is closed under taking the  restricted wreath product with a virtually infinite cyclic group. This adds to the list of known examples of Cayley automatic groups.

\end{abstract}

\maketitle

\section{Introduction}
Cayley automatic groups, introduced by Kharlampovich, Khoussainov
and Miasnikov in \cite{KKMjournal}, generalise the class of automatic groups while retaining some key algorithmic properties.  Namely, the
word problem in a Cayley automatic group is decidable in quadratic time, and the first order theory for a (directed, labeled) Cayley graph of a Cayley automatic group is decidable. 
The family of Cayley automatic groups is larger than that of automatic groups, for example it  includes all finitely generated nilpotent
groups of nilpotency class
two \cite{KKMjournal}, the Baumslag-Solitar groups \cite{BerK-BS,KKMjournal}, the
higher rank lamplighter groups \cite{Taback18}, and restricted wreath products of the form $G\wr \Z$ where $G$ is Cayley automatic  \cite{berdkhouss15}.

Here we add to this list by extending \cite{berdkhouss15} to restricted wreath products of the form $G\wr H$ where $G$ is Cayley automatic and $H$ is virtually infinite cyclic. 
While this result is not surprising, the proof contains some subtleties which require care, and we believe is worth recording.

\section{Automatic and Cayley automatic groups}
\label{sec:Auto-CGA}

We assume that the reader is familiar with the notions of regular languages, finite automata and  multi-tape synchronous automata.  For more details, we refer the reader to \cite{Epsteinbook}. We say a language $L\subseteq (X^*)^n$ is {\em regular} if it is accepted by a synchronous $n$-tape automaton where $n\in\N$ and $X$ is a finite set, or {\em alphabet}.

For any group $G$ with finite symmetric generating set $S=S^{-1}$, let $\pi\colon S^*\to G$ denote the canonical projection map. For $w\in S^*$ let $|w|_S$ denote the length of $w$ as a word in the free monoid $S^*$, that is, $|w|_S$ denotes the number of letters in the word $w$.

\begin{defn}
\label{def:aut}
An {\em automatic structure} for a group $G$ is a pair $(S,L)$ where
\be
\item $S$ is a finite symmetric generating set for $G$;
\item $L\subseteq S^*$ is a regular language;
\item  $\pi|_L \colon L \rightarrow G$ is
a bijection;
\item for each $a \in S$ the binary relation
$$R_a = \{(u,v) \in L \times L \mid \pi(u)a=_G\pi(v)\} \subseteq S^* \times S^*$$
is regular, that is, recognised by a two-tape synchronous automaton.
\ee
A group is called {\em automatic} if it has an automatic structure with respect to some finite generating set.
\end{defn} It is a standard result, see, for example \cite[Theorem 2.4.1]{Epsteinbook}, that if $G$ is automatic  then $G$ has an automatic structure with respect to any finite generating set.

Cayley automatic groups were introduced in \cite{KKMjournal} with the motivation of
allowing the language $L$ of normal forms representing group elements to be defined over an arbitrary
alphabet $\Lambda$ rather than a generating set $S$ for $G$.

\begin{defn}
\label{def:Caut}
A {\em Cayley automatic structure} for a group  $G$ is a 4-tuple  $(S,\Lambda, L,\psi)$ where
\be
\item $S$ is a finite symmetric generating set for $G$;
\item  $\Lambda$ is an alphabet and
 $L \subseteq \Lambda^*$  is a regular
language;
\item
$\psi\colon L \rightarrow G$ is a bijection;
\item
for each $a \in S$
the binary relation
$$R_a = \{(u,v) \in L \times L \,
|\,\psi(u)a=_G\psi(v)\} \subseteq \Lambda^* \times \Lambda^*$$
is regular, that is, recognised by
a two-tape synchronous automaton. \ee

A group is called {\em Cayley automatic} if it has a Cayley automatic structure   $(S,\Lambda, L,\psi)$  with respect to some finite generating set $S$.\end{defn}
As for automatic groups, if $G$ has a  Cayley automatic structure    $(S,\Lambda, L,\psi)$
and $Y$ is another finite generating set for $G$, then there exists a Cayley automatic structure    $(Y,\Lambda_Y, L_Y,\psi_Y)$ for $G$. See \cite[Theorem 6.9]{KKMjournal} for a proof of this fact.

    \section{Wreath products with virtually infinite cyclic groups}

   For two groups $G$ and $H$,
   let $G^{(H)}$ be the set of all functions $\gamma\colon H \rightarrow G$ 
   with finite support, that is,
    such that $\gamma(h)\neq 1_G$ for at most finitely many $h\in H$.
   For a given $\gamma \in G^{(H)}$ and $h \in H$, 
   we denote by $\gamma^h$ 
   the element of $G^{(H)}$ for which 
   $\gamma^h (x) = \gamma(hx)$ for all $x \in H$.
   The restricted wreath product
    $G \mathrel{\wr} H$ can be defined as the 
    Cartesian product 
   $G^{(H)} \times H$ with the group multiplication given by
   the formula:  
   \[(\gamma, h )  \cdot ( \gamma', h' )    = 
     \left(\gamma (\gamma')^{\, h^{-1}},  h h' \right).\]

     Equivalently, we can define 
$G\wr H$ as \[\left\{(\gamma,h) \mid h\in H , \gamma \in \bigoplus_{k \in H} (G)_k \text{ where $\gamma$ has finitely many nontrivial entries}\right\}\] with multiplication defined as above, 
where 
\[(\gamma')^{\, h^{-1}}=\left(\bigoplus_{k \in H} (g)_k\right)^{\, h^{-1}}=\bigoplus_{k \in H} (g)_{h^{-1}k}.\] Note that if $G$ is generated by $S_0\subseteq G$ and $H$ is generated by $T\subseteq H$ then $G\wr H$ is generated by $S_0\cup T$.

We prove the following theorem.

\begin{thm}[Wreath products with virtually infinite cyclic groups]
\label{thmC:wreath}
Let $G$ be a Cayley automatic group, and $H$ any virtually infinite cyclic group.  Then $G \wr H$ is Cayley automatic.\end{thm}

\begin{proof}
Since $G$ is Cayley automatic, there exists a finite symmetric generating set $S_0$ for $G$, 
 an alphabet   $\Lambda_0$, 
 a regular language  $L_0 \subseteq \Lambda_0^{*}$,  a bijection $\psi_0\colon L_0\to G$, and a 2-tape automaton $\texttt{M}_s$  for each $s\in S_0$ with accepted language $$L(\texttt{M}_s)=\{(u,v)\in L_0\times L_0\mid \psi_0(v)=_G \psi_0(u)s\}.$$
Without loss of generality assume $\psi_0(\varepsilon)=1_G$.

Let $H$ be a finite extension of its cyclic subgroup $\mathbb{Z} = \langle t \rangle$ of index $m+1$,
and denote by $ \langle t \rangle \h_0,  \langle t \rangle \h_1, \dots, \langle t \rangle \h_m$ the
distinct  right cosets of $\Z$, where $\h_0 =1_H$.
Let \[T = \{t, \h_1,\dots,\h_m, t^{-1}, \h_1^{-1}, \dots, \h_m^{-1}\};\]
then  $S = S_0 \cup T$ is a symmetric generating set for $G\wr H$.
We identify a particular spanning tree $\mathcal S$ of the Cayley graph $\Gamma(H,T)$ which consists of a ``spine'' corresponding to $\langle t \rangle$, and at each vertex $t^k$  there are $m$ ``spokes'' terminating at the $m$ vertices $t^kx_j$ of $H$, for $k\in \Z$ and  $1\leq j\leq m$, as in Figure~\ref{fig:spinespine}.

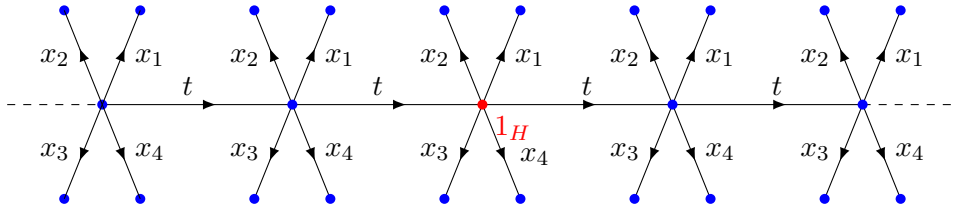
\begin{figure}[h!]
  \centering
  \begin{tikzpicture}[scale=1.25]

 \begin{scope}[decoration={markings,mark = at position 0.6 with {\arrow[scale=1.25,black]{latex}}}]

       \draw[dashed] (-5,0) -- (-4,0);
      \draw[dashed] (4,0) -- (5,0);
  
       \draw[postaction={decorate}] (-4,0) -- (-2,0);

        \draw[postaction={decorate}] (-2,0) -- (0,0);

  \draw[postaction={decorate}]
         (0,0) -- (2,0);

        \draw[postaction={decorate}]
         (2,0) -- ( 4,0);

\draw (.5,.5) node {$x_1$}; 
\draw (.55,-.55) node {$x_4$}; 
\draw (-.5,.5) node {$x_2$}; 
\draw (-.5,-.5) node {$x_3$}; 
\draw (1.1,.2) node {$t$}; 

\draw (2.5,.5) node {$x_1$}; 
\draw (2.5,-.5) node {$x_4$}; 
\draw (1.5,.5) node {$x_2$}; 
\draw (1.5,-.5) node {$x_3$}; 
\draw (3.1,.2) node {$t$}; 

\draw (4.5,.5) node {$x_1$}; 
\draw (4.5,-.5) node {$x_4$}; 
\draw (3.5,.5) node {$x_2$}; 
\draw (3.5,-.5) node {$x_3$}; 

\draw (-1.1,.2) node {$t$}; 
\draw (-2.5,.5) node {$x_2$}; 
\draw (-2.5,-.5) node {$x_3$}; 
\draw (-1.5,.5) node {$x_1$}; 
\draw (-1.5,-.5) node {$x_4$}; 
\draw (-3.1,.2) node {$t$}; 

\draw (-4.5,.5) node {$x_2$}; 
\draw (-4.5,-.5) node {$x_3$}; 
\draw (-3.5,.5) node {$x_1$}; 
\draw (-3.5,-.5) node {$x_4$}; 

      \draw[postaction={decorate}]
         (0,0) -- ( .4,1);
               \draw[postaction={decorate}]
         (0,0) -- ( .4,-1);
               \draw[postaction={decorate}]
         (0,0) -- ( -.4,1);
               \draw[postaction={decorate}]
         (0,0) -- ( -.4,-1);

   \fill[radius=1.5pt,blue]
    (.4,1) circle ;
       \fill[radius=1.5pt,blue]
    (-.4,1) circle ;
       \fill[radius=1.5pt,blue]
    (.4,-1) circle ;
   \fill[radius=1.5pt,blue]
    (-.4,-1) circle ;

               \draw[postaction={decorate}]
         (2,0) -- ( 2.4,1);
               \draw[postaction={decorate}]
         (2,0) -- ( 2.4,-1);
               \draw[postaction={decorate}]
         (2,0) -- ( 1.6,1);
               \draw[postaction={decorate}]
         (2,0) -- ( 1.6,-1);
            \fill[radius=1.5pt,blue]
    (2.4,1) circle ;
       \fill[radius=1.5pt,blue]
    (1.6,1) circle ;
       \fill[radius=1.5pt,blue]
    (2.4,-1) circle ;
   \fill[radius=1.5pt,blue]
    (1.6,-1) circle ;

      \draw[postaction={decorate}]
         (4,0) -- ( 4.4,1);
               \draw[postaction={decorate}]
         (4,0) -- ( 4.4,-1);
               \draw[postaction={decorate}]
         (4,0) -- ( 3.6,1);
               \draw[postaction={decorate}]
         (4,0) -- ( 3.6,-1);
            \fill[radius=1.5pt,blue]
    (4.4,1) circle ;
       \fill[radius=1.5pt,blue]
    (3.6,1) circle ;
       \fill[radius=1.5pt,blue]
    (4.4,-1) circle ;
   \fill[radius=1.5pt,blue]
    (3.6,-1) circle ;

      \draw[postaction={decorate}]
         (-2,0) -- ( -2.4,1);
               \draw[postaction={decorate}]
         (-2,0) -- ( -2.4,-1);
               \draw[postaction={decorate}]
         (-2,0) -- ( -1.6,1);
               \draw[postaction={decorate}]
         (-2,0) -- ( -1.6,-1);
            \fill[radius=1.5pt,blue]
    (-2.4,1) circle ;
       \fill[radius=1.5pt,blue]
    (-1.6,1) circle ;
       \fill[radius=1.5pt,blue]
    (-2.4,-1) circle ;
   \fill[radius=1.5pt,blue]
    (-1.6,-1) circle ;

   \fill[radius=1.5pt,red]
    (0,0) circle node[below right =1pt] {$1_H$};

   \fill[radius=1.5pt,blue]
    (2,0) circle ;

   \fill[radius=1.5pt,blue]
    (4,0) circle ;

   \fill[radius=1.5pt,blue]
    (-2,0) circle ;

   \fill[radius=1.5pt,blue]
    (-4,0) circle ;

            \fill[radius=1.5pt,blue]
    (-4.4,1) circle ;
       \fill[radius=1.5pt,blue]
    (-3.6,1) circle ;
       \fill[radius=1.5pt,blue]
    (-4.4,-1) circle ;
   \fill[radius=1.5pt,blue]
    (-3.6,-1) circle ;
      \draw[postaction={decorate}]
         (-4,0) -- ( -4.4,1);
               \draw[postaction={decorate}]
         (-4,0) -- ( -4.4,-1);
               \draw[postaction={decorate}]
         (-4,0) -- ( -3.6,1);
               \draw[postaction={decorate}]
         (-4,0) -- ( -3.6,-1);

     \end{scope}
\end{tikzpicture}
 \caption{Part of a spanning tree $\mathcal S$ for $\Gamma(H,T)$, where the index of $\Z=\langle t\rangle$ in $H$ is 5.}
 \label{fig:spinespine}
 \end{figure}

As a concrete example, consider the infinite dihedral group \[H=D_\infty=\langle a,b \mid a^2,b^2\rangle\] In this case we can take $t=ab, x_1=a$ and $\mathcal S= \Gamma(D_\infty,T)$, as shown in Figure~\ref{fig:spineDihedral}.

\begin{figure}[h!]
  \centering
  \begin{tikzpicture}[scale=1.25]

 \begin{scope}[decoration={markings,mark = at position 0.6 with {\arrow[scale=1.25,black]{latex}}}]

       \draw[dashed] (-5,0) -- (-4,0);
      \draw[dashed] (4,0) -- (5,0);
  
       \draw[postaction={decorate}] (-4,0) -- (-2,0);

        \draw[postaction={decorate}] (-2,0) -- (0,0);

  \draw[postaction={decorate}]
            (0,0) -- (2,0);    

        \draw[postaction={decorate}]
         (2,0) -- ( 4,0);

\draw (.25,.25) node {$a$}; 
\draw (1.1,-.2) node {$t$}; 

\draw (2.25,.25) node {$a$}; 
\draw (3.1,-.2) node {$t$}; 

\draw (4.25,.25) node {$a$}; 
\draw (-.9,-.2) node {$t$};

\draw (-1.75,.25) node {$a$}; 
\draw (-2.9,-.2) node {$t$};

\draw (-3.75,.25) node {$a$};

      \draw[postaction={decorate}]
         (0,0) -- ( .2,.5);

   \fill[radius=1.5pt,blue]
    (.2,.5) circle ;

               \draw[postaction={decorate}]
         (2,0) -- ( 2.2,.5);

            \fill[radius=1.5pt,blue]
    (2.2,.5) circle ;

      \draw[postaction={decorate}]
         (4,0) -- ( 4.2,.5);

            \fill[radius=1.5pt,blue]
    (4.2,.5) circle ;

      \draw[postaction={decorate}]
         (-2,0) -- ( -1.8,.5);

            \fill[radius=1.5pt,blue]
    (-1.8,.5) circle ;

   \fill[radius=1.5pt,red]
    (0,0) circle node[below right =1pt] {$1_{D_\infty}$};

   \fill[radius=1.5pt,blue]
    (2,0) circle ;

   \fill[radius=1.5pt,blue]
    (4,0) circle ;

   \fill[radius=1.5pt,blue]
    (-2,0) circle ;

   \fill[radius=1.5pt,blue]
    (-4,0) circle ;

            \fill[radius=1.5pt,blue]
    (-3.8,.5) circle ;
      \draw[postaction={decorate}]
         (-4,0) -- ( -3.8,.5);
%

     \end{scope}

 \begin{scope}[decoration={markings,mark = at position 0.5 with {\arrow[scale=1,gray]{latex}}}]

       \draw[dotted] (-4.8,0.5) -- (-3.8,0.5);
      \draw[dotted] (4.2,0.5) -- (5.2,0.5);
  
       \draw[dotted,postaction={decorate}] (-1.8,0.5) --(-3.8,0.5) ;

        \draw[dotted,postaction={decorate}]  (0.2,0.5)--(-1.8,0.5);

  \draw[dotted,postaction={decorate}]
       (2.2,.5) --   ( .2,.5);

        \draw[dotted,postaction={decorate}]
   ( 4.2,0.5)--      (2.2,0.5);

\draw[gray] (1.1,.7) node {$t$}; 

\draw[gray] (3.1,.7) node {$t$}; 

\draw[gray] (-.9,.7) node {$t$}; 

\draw[gray] (-2.9,.7) node {$t$};

     \draw[dotted,postaction={decorate}]
   ( 0,0) to  [bend left=45]  (0.2,0.5);
        \draw[dotted,postaction={decorate}]
   ( 2,0) to  [bend left=45]  (2.2,0.5);
        \draw[dotted,postaction={decorate}]
   ( 4,0) to  [bend left=45]  (4.2,0.5);
        \draw[dotted,postaction={decorate}]
   ( -2,0) to  [bend left=45]  (-1.8,0.5);
        \draw[dotted,postaction={decorate}]
   ( -4,0) to  [bend left=45]  (-3.8,0.5);

\footnotesize
\draw[gray] (-.2,.28) node {$a^{-1}$}; 
\draw[gray] (-2.2,.28) node {$a^{-1}$}; 
\draw[gray] (-4.2,.28) node {$a^{-1}$}; 
\draw[gray] (1.8,.28) node {$a^{-1}$}; 
\draw[gray] (3.8,.28) node {$a^{-1}$}; 

\end{scope}
\end{tikzpicture}
 \caption{Part of the spanning tree $\mathcal S$ (shown in black) drawn inside $\Gamma(D_\infty,T)$, where $T=\{t,a,t^{-1},a^{-1}\}$. Note $a=a^{-1}$ and $ata=aaba=ba=t^{-1}$. }
 \label{fig:spineDihedral}
 \end{figure}
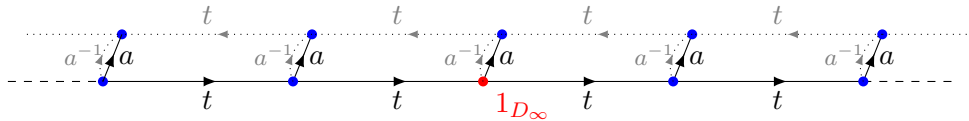

We borrow some terminology from the lamplighter groups $\Z_n \wr \Z$ to describe elements of $G \wr H$.
An element $v \in G \wr H$ can be thought of in two equivalent ways:
\be
    \item algebraically, as an element $(\gamma,h)$ where $\gamma \in \bigoplus_{k \in H} (G)_k$ has finitely many nontrivial entries and $h \in H$.
    \item geometrically, as a copy of $\mathcal S$ (or $ \Gamma(H,T)$) where each vertex is marked by some element of $G$, with all but finitely many  vertices marked by $1_G$, and the vertex $h$ of $\mathcal S$ is also marked with a {\em pointer} indicating the final position of the ``lamplighter''. We refer to this marking  as a {\em configuration} of ${\mathcal S}$.
\ee

Write $v = (\gamma,h)$ where
$\gamma \in \bigoplus_{h_i \in H} (G)_{h_i}$ and $h\in H$.
Since every element $h$ of $H$ can be written uniquely as $h = t^kx_q$ for some $k\in \Z$ and $0\leq q\leq m$, the map 
 $\xi:H \rightarrow \Z$ defined by $\xi(t^kx_q)=k$ is well defined.
Then the vertex corresponding to $h \in H$ is an endpoint of a spoke attached to the vertex $t^{\xi(h)}$.
For $v = (\gamma,h)$ with $\gamma$ as above, let
\be
    \item $k_* = \xi(h)$,
    \item $k_1 = \min\{0,\xi(h_i) \mid  (g)_{h_i} \in \gamma , (g)_{h_i}\neq 1_G\}$, and
    \item $k_2 = \max\{0,\xi(h_i) \mid  (g)_{h_i} \in \gamma, (g)_{h_i}\neq 1_G \}$.
\ee
Additionally, let $m_1 = \min(k_*,k_1)$ and $m_2 = \max(k_*,k_2)$.
Define the {\em integer support} of $v$, 
denoted $\sppt(v)$,  to be the interval $[m_1,m_2]$.
The left endpoint of the integer support is the smallest $k$ so that either
\be
    \item $v$ has a nontrivial entry among the copies of $\Gamma(G,S_0)$ attached to the spine 
    at the vertex $t^k x_i$ for some 
    $0 \leq i \leq m$, 
    \item the final position of the lamplighter is $t^kx_i$ for some $0 \leq i \leq m$, or
    \item all of $k_*$ and $\xi(h_i)$ are positive, that is, the lamplighter is never in a position along the spine with negative index, so $m_1=0$ denotes the starting position of the lamplighter.
\ee
The right endpoint of the integer support is defined analogously, where the $0$ is included in the definition of $k_2$ to account for the possibility that $k_*$ and all the $\xi(h_i)$ are negative.

To define our normal form, we
mimic the standard ``left-first'' representation of elements of the lamplighter group $\Z_n \wr \Z$ (cf.\   \cite{Taback03}).
Given $v = (\gamma,h) \in G \wr H$, we describe a path traversed by the lamplighter from the vertex $1_H$ in ${\mathcal S}$ to its final vertex $h \in {\mathcal S}$.
If $m_1<0$, the lamplighter first moves left along the spine of ${\mathcal S}$ to the vertex labeled $t^{m_1}$, and marks it with a possibly trivial element of $G$.
The lamplighter then visits $t^{m_1}x_1$ and marks it with a possibly trivial element of $G$ and returns to $t^{m_1}$.
This procedure is repeated for the vertices $t^{m_1}x_2, \cdots ,t^{m_1}x_m$.
The lamplighter then proceeds to the vertex corresponding to $t^{m_1+1}$ and repeats the process of visiting the vertex at the end of each spoke in order and marking it with a possibly trivial element of $G$.
This continues until the lamplighter reaches the vertex corresponding to $t^{m_2}$, where the process is repeated one last time.
If    $m_1 = 0$, 
the lamplighter begins the process of marking the vertices with possibly trivial elements of $G$ at $1_H \in {\mathcal S}$, and then visits the spokes as described above, until it reaches the vertex labeled $t^{m_2}$ and marks the vertices $t^{m_2}x_j$ for $0 \leq j \leq m$ with possibly trivial elements of $G$.

We refer to the subpath which starts at the vertex 
$t^{m_1}$
and ends at the vertex $t^{m_2}$ after having marked the vertices $t^{i}x_j$ for $m_1\leq i\leq m_2,0 \leq j \leq m$ with possibly trivial elements of $G$
as the {\em positive path}, because when written as a word in the group generators, the exponents of $t$ are all positive.
See Figure~\ref{fig:spine2} for an example of a configuration with integer support $[-2,1]$ where the positive path is  marked.

\begin{figure}[h!]
  \centering
  \begin{tikzpicture}[scale=1.2]

 \begin{scope}[decoration={markings,mark = at position 1 with {\arrow[scale=1.25,red]{latex}}}]

   \draw[postaction={decorate},red,thick] (-2.4,-1.45) -- (-2.4,-1.15);
       \end{scope}

 \begin{scope}[ decoration={markings,mark = at position 0.6 with {\arrow[scale=1.25,black]{latex}}}]
  
       \draw[dashed] (-5,0) -- (-4,0);
      \draw[dashed] (4,0) -- (5,0);
      
               \draw[postaction={decorate}]
         (2,0) -- ( 4,0);

      \draw (-4.5,.5) node {$x_2$}; 
\draw[red] (-4.6,1) node {$g_1$}; 
\draw (-4.5,-.5) node {$x_3$}; 
\draw (-3.5,.5) node {$x_1$}; 
\draw (-3.5,-.5) node {$x_4$};  
\draw[red] (-3.35,-1) node {$g_2$}; 

\draw (3.1,.2) node {$t$}; 
\draw (4.5,.5) node {$x_1$};  
\draw (4.5,-.55) node {$x_4$}; 
\draw (3.5,.5) node {$x_2$}; 
\draw (3.5,-.5) node {$x_3$};

               \draw[postaction={decorate}]
         (4,0) -- ( 4.4,1);
               \draw[postaction={decorate}]
         (4,0) -- ( 4.4,-1);
               \draw[postaction={decorate}]
         (4,0) -- ( 3.6,1);
               \draw[postaction={decorate}]
         (4,0) -- ( 3.6,-1);
            \fill[radius=1.5pt,blue]
    (4.4,1) circle ;
       \fill[radius=1.5pt,blue]
    (3.6,1) circle ;
       \fill[radius=1.5pt,blue]
    (4.4,-1) circle ;
   \fill[radius=1.5pt,blue]
    (3.6,-1) circle ;

          \end{scope}

 \begin{scope}[thick, color=green4, decoration={markings,mark = at position 0.6 with {\arrow[scale=1.25,green4]{latex}}}]

      \draw[postaction={decorate}]
         (-4,0) -- ( -4.4,1);
               \draw[postaction={decorate}]
         (-4,0) -- ( -4.4,-1);
               \draw[postaction={decorate}]
         (-4,0) -- ( -3.6,1);
               \draw[postaction={decorate}]
         (-4,0) -- ( -3.6,-1);

       \draw[postaction={decorate}] (-4,0) -- (-2,0);

        \draw[postaction={decorate}] (-2,0) -- (0,0);

  \draw[postaction={decorate}]
         (0,0) -- (2,0);

\draw (.5,.5) node {$x_1$}; 
\draw (.55,-.55) node {$x_4$}; 
\draw (-.5,.5) node {$x_2$}; 
\draw (-.5,-.5) node {$x_3$}; 
\draw (1.1,.2) node {$t$}; 

\draw (2.5,.5) node {$x_1$};  
\draw[red] (2.65,1) node {$g_4$};
\draw[red] (2.3,-.2) node {$g_3$};
\draw (2.5,-.55) node {$x_4$}; 
\draw (1.5,.5) node {$x_2$}; 
\draw (1.5,-.5) node {$x_3$};


\draw (-1.1,.2) node {$t$}; 
\draw (-2.5,.5) node {$x_2$}; 
\draw (-2.5,-.5) node {$x_3$}; 
\draw (-1.5,.5) node {$x_1$}; 
\draw (-1.5,-.5) node {$x_4$}; 
\draw (-3.1,.2) node {$t$};

      \draw[postaction={decorate}]
         (0,0) -- ( .4,1);
               \draw[postaction={decorate}]
         (0,0) -- ( .4,-1);
               \draw[postaction={decorate}]
         (0,0) -- ( -.4,1);
               \draw[postaction={decorate}]
         (0,0) -- ( -.4,-1);
   \fill[radius=1.5pt,blue]
    (.4,1) circle ;
       \fill[radius=1.5pt,blue]
    (-.4,1) circle ;
       \fill[radius=1.5pt,blue]
    (.4,-1) circle ;
   \fill[radius=1.5pt,blue]
    (-.4,-1) circle ;

               \draw[postaction={decorate}]
         (2,0) -- ( 2.4,1);
               \draw[postaction={decorate}]
         (2,0) -- ( 2.4,-1);
               \draw[postaction={decorate}]
         (2,0) -- ( 1.6,1);
               \draw[postaction={decorate}]
         (2,0) -- ( 1.6,-1);
            \fill[radius=1.5pt,red]
    (2.4,1) circle ;
       \fill[radius=1.5pt,blue]
    (1.6,1) circle ;
       \fill[radius=1.5pt,blue]
    (2.4,-1) circle ;
   \fill[radius=1.5pt,blue]
    (1.6,-1) circle ;

      \draw[postaction={decorate}]
         (-2,0) -- ( -2.4,1);
               \draw[postaction={decorate}]
         (-2,0) -- ( -2.4,-1);
               \draw[postaction={decorate}]
         (-2,0) -- ( -1.6,1);
               \draw[postaction={decorate}]
         (-2,0) -- ( -1.6,-1);
                     \fill[radius=1.5pt,blue]
    (-2.4,1) circle ;
       \fill[radius=1.5pt,blue]
    (-1.6,1) circle ;
       \fill[radius=1.5pt,red]
    (-2.4,-1) circle ;
   \fill[radius=1.5pt,blue]
    (-1.6,-1) circle ;

            \fill[radius=1.5pt,red]
    (-4.4,1) circle ;
       \fill[radius=1.5pt,blue]
    (-3.6,1) circle ;
       \fill[radius=1.5pt,blue]
    (-4.4,-1) circle ;
   \fill[radius=1.5pt,red]
    (-3.6,-1) circle ;

   \fill[radius=1.5pt,blue]
    (0,0) circle node[below right =1pt] {$1_H$};

   \fill[radius=1.5pt,red]
    (2,0) circle ; 

   \fill[radius=1.5pt,blue]
    (-2,0) circle ;

   \fill[radius=1.5pt,blue]
    (-4,0) circle ;

     \end{scope}
\end{tikzpicture}
 \caption{The element
 $t^{-2}x_2g_1x_2^{-1}x_4g_2x_4^{-1}t^3g_3x_1g_4x_1^{-1}t^{-2}x_3$
 as a configuration of $\mathcal S$. The integer support of this element is $[-2,1]$ and the \textcolor{red}{red} arrow denotes the final position of the lamplighter. The positive path for this element is 
 $x_1 x_1^{-1}x_2g_1x_2^{-1}x_3x_3^{-1}x_4g_2x_4^{-1}(tx_1x_1^{-1}x_2x_2^{-1}x_3x_3^{-1}x_4x_4^{-1})^2tg_3x_1g_4x_1^{-1}x_2x_2^{-1}x_3x_3^{-1}x_4x_4^{-1}$ 
  is shown in \textcolor{green4}{green}.}
 \label{fig:spine2}
 \end{figure}

Upon completing the positive path, one of two things will occur.  It may be that the lamplighter is in its final position, and the path simply ends.
If not, the lamplighter moves to its final position via a subpath of the form $t^k$ or  $t^kx_q$ where $k\in \Z, k\leq 0$.
Note that since $m_2$, the right endpoint of $\sppt(v)$, is the maximum of $k_2$ and $k_*$, the lamplighter will never be in a position along $\Z = \langle t \rangle$ to the right of $m_2$, so the exponent $k$ is non-positive.

As the lamplighter travels along its positive path, we  will wish to  indicate two special positions: the first time the lamplighter is at the vertex corresponding to $1_H=t^0$, and the first time the lamplighter is at the vertex which will be its final position.
The integer support and the positive path are defined so that these are unique positions along the positive path.

The normal form for the Cayley automatic structure on $G \wr H$ will be constructed in stages.
We first define a normal form $\mathcal N_0\subseteq (\Lambda_0\cup T)^*$ for elements of $G\wr H$ as follows.  Given $v \in G\wr H$ with $\sppt(v) = [m_1,m_2]$, the above description allows us to uniquely represent $v$ as a word either of the form
\begin{equation}
\label{eqn:nf}
v=t^nx_q\ \text{  or  }\
v=t^{m_1}v_1tv_2\dots tv_st^jx_q
\end{equation}
where $n,j,q, s \in\Z, \ j\leq 0$,
$s\geq 1,  \ 0\leq q\leq m$,  $m_1+(s-1)=m_2$
and
\begin{equation}
\label{eqn:nf2}
v_\cC=v_{\cC,0}\h_1v_{\cC,1}\h_1^{-1}\h_2v_{\cC,2}\dots \h_{m-1}^{-1}\h_mv_{\cC,m}\h_m^{-1}
\end{equation}
with $v_{\cC,t} \in L_0$.  If $m_1 = k_*$ then we allow $v_1$ to be trivial, otherwise $v_1$ must be nontrivial.
If $m_2 = k_*$ we allow $v_\cC$ to be trivial, otherwise $v_\cC$ must be nontrivial.
Each word $v_\cC$ encodes a sequence of words $(v_{\cC,0},\dots, v_{\cC,m})\in  L_0^{m+1}$ with $\psi_0(v_{\cC, 0})$ labeling the vertex at position $t^{m_1+\cC-1}$ in $\mathcal S$
and $\psi_0(v_{\cC, i})$ labeling the end of the spoke at position  $t^{m_1+\cC-1}x_i$ for $1 \leq i \leq m$.
Note that in Equation  \eqref{eqn:nf}, the $v_\cC$ are separated by instances of $t$ as the lamplighter moves along the positive path.
Let ${\mathcal N}_0\subseteq (\Lambda_0\cup T)^*$ denote the set of words of this form.

For example,
the element in
 Figure~\ref{fig:spine2}
 has $\mathcal N_0$ normal form
    \begin{equation*}
   \begin{split}
t^{-2}x_1x_1^{-1}x_2v_{1,2}x_2^{-1}x_3x_3^{-1}x_4v_{1,4}x_4^{-1}
 tx_1x_1^{-1}x_2x_2^{-1}x_3x_3^{-1}x_4x_4^{-1}
 tx_1x_1^{-1}x_2x_2^{-1} \\
 x_3x_3^{-1}x_4x_4^{-1}
  tv_{4,0}x_1v_{4,1}x_1^{-1}x_2x_2^{-1}x_3x_3^{-1}x_4x_4^{-1}
 t^{-2}x_3
   \end{split}
   \end{equation*}  where $\psi_0(v_{1,2})=g_1, \psi_0(v_{1,4})=g_2$, $\psi_0(v_{4,0})=g_3$, $\psi_0(v_{4,1})=g_4$, and  in all other cases, $v_{i,j}=\varepsilon$ where $\psi_0(\varepsilon)=1_G$.

Next, we  insert special symbols into the words in $\mathcal N_0$ to obtain the intermediate language $\mathcal N_1$.

Let
$\Lambda_1=\Lambda_0\cup T\cup\{B,C, B_0,C_*\}$ and
$\Lambda=\Lambda_0\cup\{B,C, B_0,C_*\}$.
Let $v = (\gamma,h) \in G \wr H$ be written in the form of Equation~\eqref{eqn:nf}.
Notice that all terms of the form $v_\cC$ are part of the positive path.
With $v_\cC$ as in Equation~\eqref{eqn:nf2}, before each $v_{\cC,j}$ we place the symbol $C$, with one exception.
If $v_{\cC,j}$ is the label of the vertex $h$ of {${\mathcal S}$} which is the final position of the lamplighter, then precede $v_{\cC,j}$ by the symbol $C_*$.
Before each term $v_\cC$ we place the symbol $B$, with one exception.
If $m_1+\cC-1 = 0$ we place the symbol $B_0$ in front of $v_{\cC}$, indicating the unique position along the positive path where the lamplighter is at  the vertex $1_H \in {\mathcal S}$.

Let ${\mathcal N}_1\subseteq \Lambda_1^*$ denote the set of all
words in ${\mathcal N_0}$ where the symbols $\{B,C, B_0,C_*\}$ have been inserted as described.
The word  in $\mathcal N_1$ for
the element in
 Figure~\ref{fig:spine2}
is then
\begin{equation*}
  \begin{split}
t^{-2}BCx_1Cx_1^{-1}x_2Cv_{1,2}x_2^{-1}x_3Cx_3^{-1}x_4Cv_{1,4}x_4^{-1}
 tBCx_1Cx_1^{-1}x_2Cx_2^{-1}x_3C_*x_3^{-1} \\
 x_4Cx_4^{-1} tB_0Cx_1Cx_1^{-1}x_2Cx_2^{-1}x_3Cx_3^{-1}x_4Cx_4^{-1}
  tBCv_{4,0}x_1Cv_{4,1}x_1^{-1}x_2Cx_2^{-1} \\ x_3Cx_3^{-1}
  x_4Cx_4^{-1}
 t^{-2}x_3.
   \end{split}
   \end{equation*}

To obtain the final normal form which will be the basis of the Cayley automatic structure for $G \wr H$, let ${\mathcal N}\subseteq \Lambda^*$ denote the set of words in ${\mathcal N}_1$ where all instances of the letters $t, \h_1,\dots,\h_m, t^{-1}, \h_1^{-1}, \dots, \h_m^{-1}$ from the set $T$ are removed.
The word in $\mathcal N$ for
the element in
 Figure~\ref{fig:spine2}
is then
    \begin{dmath*}
BCCCv_{1,2}CCv_{1,4}BCCCC_*C
 B_0CCCCC
  BCv_{4,0}Cv_{4,1}CCC.
   \end{dmath*}

Define the language
\begin{dmath}\label{product}
L_1=
 \left\{\prod_{i=1}^p\left(\beta_i\Gamma_{i,0}v_{i,0}\Gamma_{i,1}v_{i,1}\Gamma_{i,2}v_{i,2}\dots \Gamma_{i,m}v_{i,m}\right)\ \middle|\ \begin{array}{ll}
 v_{i,j}~\in L_0,\\
\beta_i~\in\{B,B_0\}, \\ \Gamma_{i,j}~\in \{C,C_\ast \},\\
p\geq 1\end{array}\right\}.
\end{dmath}
As  $L_0$ is a regular language, it follows that $L_1$ is a regular language.

Recall that when $v \neq t^kx_q$,
if $m_1 = k_*$ we allow $v_1$ to be trivial, otherwise $v_1$ must be nontrivial,
and if $m_2 = k_*$ we allow $v_s$ to be trivial, otherwise $v_s$ must be nontrivial.
These conditions are easily verified by a finite state automaton  inspecting, respectively, the first and last expressions in the product representing an element of $L_1$ according to the following rules:
\begin{enumerate}
    \item if $C_*$ occurs in the first factor in the product as expressed in Equation~\ref{product}, then all $v_{i,j}$ may be $\varepsilon$ for $0 \leq j \leq m$; if not, at least one $v_{i,j}$ must be nontrivial.
    \item if the $C_*$ occurs in the last factor in the product as expressed in Equation~\ref{product}, then all $v_{i,j}$ may be $\varepsilon$ for $0 \leq j \leq m$; if not, at least one $v_{i,j}$ must be nontrivial.
\end{enumerate}
Note that a  finite state automaton can also easily  verify that when all $v_\cC$ are trivial, we have a normal form corresponding to $t^kx_q$.  Let $L_2\subseteq L_1$ be the set of all strings in $L_1$ for which all three of these conditions are satsified.
Since these conditions are easily checked by a finite state automaton, and $L_1$ is a regular language, 
 it follows that $L_2$ is a regular language.

Finally, we verify that the string has only one occurrence each of $B_*$ and $C_*$. Let
\[ L=L_2\cap \{pB_0qC_\ast r, pC_\ast qB_0r\mid p,q,r\in (\Lambda\setminus\{B_0,C_*\})^*\} . \]
It follows that $L$ is regular and that $L = {\mathcal N}$.

As a further example,
note that if $v = t^nx_q$,
the corresponding word in $L$ is  as follows:
\be
\item when $n> 0$, we have $\sppt(t^nx_q) = [0,n]$
and the corresponding  word is \[B_0C^{m+1}(BC^{m+1})^{n-1}B C^qC_*C^{m-q};\]

\item when $n=0$, we have $\sppt(x_q) = [0,0]$
and the corresponding word is \[B_0 C^qC_*C^{m-q};\]

\item when $n < 0$, we have $\sppt(t^nx_q) = [n,0]$ and the corresponding  word is \[B
C^qC_*C^{m-q}
(BC^{m+1})^{n-1}B_0C^{m+1}.\]
\ee

Given a word $\sigma \in L$, the symbols $B_0$ and $C_*$ allow us to reconstruct the integer support of the corresponding element, as well as the final position of the lamplighter, that is, the coordinate $h$.
The words $v_{i,j}$ correspond (via $\psi_0$) to elements of $G$ listed in a specified order.  That is, we can deterministically reconstruct $\gamma \in \bigoplus_{h \in H} (G)_h$ and $h \in H$ from $\sigma$.
Formally, let $\psi\colon L \rightarrow G\wr H$ be the bijective map defined by
\begin{equation*}
\psi(w) = \psi\left(\prod_{k=1}^s\left(\beta_k\Gamma_{k,0}u_{k,0}\Gamma_{k,1}u_{k,1}\Gamma_{k,2}u_{k,2}\dots \Gamma_{k,m}u_{k,m}\right)\right)  = t^{m_1}p_1tp_2\dots tp_st^j
\end{equation*}
where
\[p_i=\psi_0(v_{i,0})\h_1\psi_0(v_{i,1})\h_1^{-1}\h_2\psi_0(v_{i,2})\dots \h_{m-1}^{-1}\h_m\psi_0(v_{i,m})\h_m^{-1}, \]
with $u_{i,j} = v_{i,j}$, $\beta_k \in \{B,B_0\}$, $\Gamma_{k,j} \in \{C,C_*\}$ and $m_1$ calculated from the positions of $B_0$ and $C_*$ as described above.

We claim that $(S,\Lambda,L,\psi)$ is a Cayley automatic structure for $G \wr H$.
To prove this,
we must show that for every generator $s \in S=S_0\cup T$ the set 
\[R_s = \{(u,v) \in L \times L | \psi(u)s =_{G \wr H} \psi(v) \}\] is a regular language, that is, recognised by a 2-tape synchronous automaton.
It suffices to do this for $s \in S_0 \cup \{x_1, \cdots ,x_m,t\}$; see, for example, \cite[Lemma 9]{cga}.

First let $s \in S_0$ and 
suppose $(u,v) \in R_s$.
Viewing $\psi(u)$ as a {configuration} of ${\mathcal S}$ with finitely many vertices marked with elements of $G$
and a distinguished position for the lamplighter,
we can easily see the effect of multiplication  by $s$ on the normal form.
Let $t^kx_q$ denote the vertex of ${\mathcal S}$ which is the final position of the lamplighter in $u$, marked by the element  $g_u \in G$. Let $\rho_u\in L_0$ be such that $\psi_0(\rho_u)=g_u$.
To obtain the normal form word for $\psi(u)s$ we simply multiply $\rho_u$ by $s$ and verify that the multiplication is correct using the multiplier automaton $\texttt{M}_s$ given as part of the given Cayley automatic structure on $G$.
Therefore we need to accept pairs of strings $(u,v) \in L \times L$ of the following form:
\begin{dmath*}
u = \left(\Pi_{i=1}^{p}  \beta_i \Pi_{j=0}^m C \alpha_{i,j}\right)
\Theta_u
\left(\Pi_{i=p+2}^{\cC}  \beta_i   \Pi_{j=0}^m C \alpha_{i,j}\right)
\end{dmath*}
and
\begin{dmath*}
v = \left(\Pi_{i=1}^{p}  \beta_i   \Pi_{j=0}^m C \alpha_{i,j}\right)
\Theta_v
\left(\Pi_{i=p+2}^{\cC}  \beta_i  \Pi_{j=0}^m C \alpha_{i,j}\right)
\end{dmath*}
where $\beta_i \in \{B,B_0\}$, $\alpha_{i,j} \in L_0$,
\[
\Theta_u =  \beta_{p+1}  C\alpha_{p+1,0} \cdots C_* \bm{\alpha_{p+1,r}} \cdots C\alpha_{p+1,m}
\]
and
\[
\Theta_v = \beta_{p+1}  C\alpha_{p+1,0} \cdots C_* \bm{\alpha'_{p+1,r}} \cdots C\alpha_{p+1,m}
\]
where $(\alpha_{p+1,r},\alpha'_{p+1,r})$ is accepted by the multiplier automaton $\texttt{M}_s$ given as part of the given Cayley automatic structure on $G$.  The bold highlighted symbols represent the only difference between the two words.

By
\cite{KKMjournal} (see also \cite[Lemma 8]{cga})  the  language $L_0$ is necessarily quasigeodesic.
It follows that the difference between the lengths of
$\alpha_{p+1,r}$ and $ \alpha'_{p+1,r}$ is uniformly bounded.
As it is regular to check that two words are identical with a bounded shift, 
it follows that we can construct a 2-tape automaton which checks that the prefixes  of $u$ and $v$ are identical, then calls   $\texttt{M}_s$ to read $(\alpha_{p+1,r},\alpha'_{p+1,r})$, and finally checks that the suffixes of $u$ and $v$ are identical (with a bounded shift).
Thus 
 $R_s$ is a regular language.

Next let $x_i \in \{x_1, \cdots ,x_m\}$, and 
suppose $(u,v) \in R_{x_i}$.
Writing $\psi(u)$ as in Equation~\eqref{eqn:nf}, we see that $\psi(u)x_i$ ends in the letters $x_qx_i$.
The product $x_qx_i \in H$ is an element of some right coset $\langle t \rangle x_r$.  That is, $x_qx_i = t^kx_r$ for some $k$ and $r$.
Viewing $\psi(u)$ and $\psi(v)$ as configurations in ${\mathcal S}$, this means that the configurations are identical except for the final position of the lamplighter which is indicated by $C_*$ in the normal form.
Note that as $x_q$ and $x_i$ vary among the finite set of coset representatives, there are only a finite number of possible values of $(k,r)$ which arise.

The elements $\psi(u)$ and $\psi(v)$ may or may not have identical integer support. For example 
if  $\psi(u)=t^{-10}g_1t^{20}g_2t^{-5}x_q$ and $x_qx_i=t^{-7}x_r$ then $\sppt(\psi(u))=\sppt(\psi(v))$, 
whereas if  $\psi(u)=t^{-10}g_1t^{20}g_2t^{-5}x_q$ and $x_qx_i=t^{-17}x_r$ then $\sppt(\psi(u))\neq\sppt(\psi(v))$.

If $\sppt(\psi(u)) = \sppt(\psi(v))$,  then we simply need to check the two strings are identical except for the location of $C_*$.  
If $\psi(u)$ ends in $x_q$ when written as in Equation~\eqref{eqn:nf},  we have  $x_qx_i = t^kx_r$.  
Let $\pi\colon\Lambda^*\to \{C,C_*,B_0\}^*$ be a homomorphism which is the identity on $C,C_*,B_0$ and sends all other letters to $\varepsilon$. 
Then $\pi(u)$ and $\pi(v)$ are identical strings except for the location of $C_*$ in each string. Observe the letter $B_0$ is in the same position in each string since the integer supports of $\psi(u)$ and $\psi(v)$ are the same.
Further observe that there exists an integer  $s_{q,i}$ such that for every pair $(u,v)\in R_{x_i}$ which have the same integer support, if $C_*$ is the $x$th letter of $\pi(u)$ and the $y$th letter of $\pi(v)$, then $x-y=s_{q,i}$.

 Consider the language $X_{q,i}\subseteq \{C,C_*,B_0\}^*\times  \{C,C_*,B_0\}^*$
consisting of all pairs of strings, each of which contains exactly one $C_*$ letter and one $B_0$ letter, 
where  $C_*$ is the $x$th letter of the first string and the $y$th letter of the second string with $x-y=s_{q,i}$,
and $B_0$ is in the same position in both strings.  
Since these conditions are regular to check, $X_{q,i}$ is a regular language.

Let \[\kappa\colon \Lambda^*\times \Lambda^*\to  \{C,C_*,B_0\}^*\times  \{C,C_*,B_0\}^*\] be the map which  in each coordinate is the identity on  $C,C_*$ and $B_0$ and sends all other letters  to $\varepsilon$.  
Let $Y \subseteq \Lambda^* \times \Lambda^*$ be the 
language consisting of all pairs of strings 
such that for every positive integer $z$ the $z$th letter of the first
string and the second string is the same unless one of these 
letters is $C_*$ and the other is $C$. The language $Y$ is regular.    
Then the language \[\kappa^{-1}(X_{q,i})\cap \left(L\times L\right) \cap Y\] is regular, and the union of these languages for $0\leq q\leq m+1$   is exactly the subset of $R_{x_i}$ for which multiplication by $x_i$ does not change the integer support for the first entry.

Now consider all the possible ways that the integer support of $\psi(u)$ can change upon multiplication by $x_i$. Again assume $\psi(u)$  ends in $x_q$ when written as in Equation~\eqref{eqn:nf}, and
$x_qx_i = t^kx_r$.  We must consider the following cases.

\begin{enumerate}\item $  \psi(u)=t^nx_q$ and $k\neq 0$,
\item $\psi(u)=t^{m_1}v_1tv_2\dots tv_st^jx_q$ with  $\sppt(\psi(u))=[m_1,m_2], j\leq 0$ and
\begin{enumerate}\item $k>-j$; in this case  the integer support  of $v$ extends  further to the right of $m_2$,
\item $k<m_1-m_2-j$; in this case the  integer support of $v$  extends  further to the left of $m_1$.
\end{enumerate}

\end{enumerate}
Each of these cases can be handled in a manner similar to the above case, by considering the relative positions of $C_*$ and $B_0$ in $\pi(u),\pi(v)$. 
For the first case, if $n>0,k>-n$ then 
 \begin{equation*}
    \begin{split}
    u= B_0C^{m+1}(BC^{m+1})^{n-1}B C^qC_*C^{m-q} \\\text{and } v=B_0C^{m+1}(BC^{m+1})^{n+k-1}B C^rC_*C^{m-r};  \end{split}
    \end{equation*}
if $n>0,k<-n$ then
 \begin{equation*}
    \begin{split}
    u  =  B_0C^{m+1}(BC^{m+1})^{n-1}B C^qC_*C^{m-q}  \\ \text{and } v = B
C^rC_*C^{m-r}
(BC^{m+1})^{-k-n-1}B_0C^{m+1}
    \end{split}
    \end{equation*}
 Analogous pairs of expressions can be worked out for $n\leq 0$; clearly all such pairs can be recognised by 2-tape automata  since $q,i,k,r$ are fixed.
We leave details of the remaining cases to the reader.

Finally, 
suppose $(u,v) \in R_{t}$.
Writing $\psi(u)$ as in Equation~\eqref{eqn:nf}, we see that $\psi(u)t$ ends in the letters $x_qt$. Once again, we can consider the case where the integer support of $\psi(u)$ does not change, in which case we merely need to check the location of the $C_*$ letters in each word, 
and separately the case where the integer support of $\psi(u)$ differs at one endpoint from the integer support of $\psi(v)$.
We follow the same reasoning as in the previous case of multiplication by $x_i$; note that $x_qt$ is in some right coset of $\Z$ in $H$, so we can write $x_qt = t^kx_r$, for a possibly different coset representative $x_r$.
We can  therefore show that  $R_t$ is a regular language as well.
The regular languages $R_s$, $R_{x_i}$ and $R_t$ complete the construction of the Cayley automatic structure
$(S,\Lambda,L,\psi)$ for $G \wr H$. 
\end{proof}

\bibliographystyle{plain}
\bibliography{BET_bibliography}

\end{document}